\documentclass[
final,
]{amsart}

\usepackage[english]{babel}
\usepackage[utf8]{inputenc}

\usepackage[T1]{fontenc}
\usepackage{lmodern}

\usepackage{enumitem}
\usepackage{verbatim}
\usepackage{url}
\usepackage[notref,notcite,color]{showkeys}
\usepackage{microtype}
\usepackage{varwidth} 

\usepackage[LogicalSystemsSans]{ak_logic}

\newcommand*{\Theorem}{Theorem}
\newcommand*{\Proposition}{Proposition}
\newcommand*{\Lemma}{Lemma}
\newcommand*{\Corollary}{Corollary}
\newcommand*{\Definition}{Definition}
\newcommand*{\Remark}{Remark}
\newcommand*{\Notation}{Notation}

\theoremstyle{plain}
\newtheorem{theorem}{\Theorem}

\newtheorem{lemma}[theorem]{\Lemma}
\theoremstyle{definition}
\newtheorem{definition}[theorem]{\Definition}
\theoremstyle{remark}
\newtheorem{remark}[theorem]{\Remark}

\usepackage{prettyref}
\newrefformat{thm}{\Theorem~\ref{#1}}
\newrefformat{pro}{\Proposition~\ref{#1}}
\newrefformat{cor}{\Corollary~\ref{#1}}
\newrefformat{lem}{\Lemma~\ref{#1}}
\newrefformat{rem}{\Remark~\ref{#1}}
\newrefformat{def}{\Definition~\ref{#1}}

\renewcommand{\epsilon}{\varepsilon}

\newcommand{\U}{\mathcal{U}}
\newcommand{\Uidem}{\U_{\textup{\rm idem}}}
\newcommand{\Uss}{\U_{\textup{\rm ss}}}

\newcommand{\F}{\mathcal{F}}
\newcommand{\ps}[1]{\mathcal{P}(#1)}

\newcommand{\T}{\mathbf{T}}

\newcommand{\FS}[1]{{\mathrm{FS}(#1)}}
\newcommand{\FFS}[1]{{\mathcal{F}(#1)}}

\DeclareMathSymbol{\two}{\mathalpha}{letters}{`2}

\title{On idempotent ultrafilters in higher-order reverse mathematics}

\author{Alexander P.\@ Kreuzer}
\address{ENS Lyon, Université de Lyon, LIP (UMR 5668 -- CNRS -- ENS Lyon -- UCBL -- INRIA) \\
46~Allée d'Italie,
69364 Lyon Cedex 07, France
}
\email{alexander.kreuzer@ens-lyon.fr}
\urladdr{http://perso.ens-lyon.fr/alexander.kreuzer/}
\thanks{The author was supported by the Récré project and by the German Science Foundation (DFG Project KO 1737/5-1).}

\date{\today}

\begin{document}

\begin{abstract}
  We analyze the strength of the existence of idempotent ultrafilters in higher-order reverse mathematics.

  Let \lpf{\Uidem} be the statement that an idempotent ultrafilter on $\Nat$ exists.
  We show that over \ls{ACA_0^\omega}, the higher-order extension  of \ls{ACA_0}, the statement \lpf{\Uidem} implies the iterated Hindman's theorem (\lp{IHT}) and we show that $\ls{ACA_0^\omega}+\lpf{\Uidem}$ is $\Pi^1_2$-conservative over $\ls{ACA_0^\omega}+\lp{IHT}$ and thus over \ls{ACA_0^{+}}.
\end{abstract}

\maketitle

In \cite{aK12b} we developed a technique to extract programs of proofs using non-\hspace{0pt}principal ultrafilters. Along these lines, we also proved a conservativity result for the statement that a non-principal ultrafilter exists.

In this paper we apply this technique to idempotent ultrafilters. 
An idempotent ultrafilter $\U$ (over $\Nat$) is an ultrafilter such that $\U=\U+\U$ where the addition is given by
\[
\U + \mathcal{V} = \{\, X \subseteq \Nat \mid \{\,n\in \Nat \mid X-n \in \mathcal{V} \,\} \in \mathcal{U}\,\}
.\]
The set of all ultrafilters on $\Nat$ can be identified with the Stone-\v{C}ech compactification of $\Nat$.
One can show that the addition defined above is the extension of the addition of $\Nat$ to $\beta\Nat$. Together with it $\beta\Nat$ becomes a left topological compact semigroup. The existence of idempotent elements follows then from Ellis' Theorem. See \cite{HS12,vB10} for an overview. 

We will show that the existence of idempotent ultrafilters is $\Pi^1_2$\nobreakdash-\hspace{0pt}conservative over the iterated Hindman's theorem (\lp{IHT}) as defined by Blass, Hirst, Simpson in \cite{BHS87}, see also \cite{jH04}.

Let \lpf{\U}, \lpf{\Uidem} be the statements that a non-\hspace{0pt}principal resp.\@ idempotent non-\hspace{0pt}principal ultrafilter on $\Nat$ exists. Let \ls{RCA_0^\omega}, \ls{ACA_0^\omega} be the extensions of \ls{RCA_0} resp.\@ \ls{ACA_0} to higher-order arithmetic as introduced by Kohlenbach in \cite{uK05b}. In \ls{RCA_0^\omega} or \ls{ACA_0^\omega} the statements \lpf{\U}, \lpf{\Uidem} can be formalized using an object of type $\Nat^\Nat \longrightarrow \Nat$.

Further, let Feferman's $\mu$ be a functional of type $\Nat^\Nat\longrightarrow \Nat$ satisfying 
\[
f(\mu(f))= 0\quad\text{if}\quad \Exists{x} f(x)=0 ,
\]
and let \lpf{\mu} be the statement that such a functional exists. Clearly, \lpf{\mu} implies arithmetical comprehension. However, $\mu$ is not definable in \ls{ACA_0^\omega}.

In \cite{aK12b} we showed that
\begin{itemize}
\item $\ls{RCA_0^\omega} \vdash \lpf{\U} \IMPL \lpf{\mu}$ and that
\item $\ls{ACA_0^\omega}+\lpf{\mu}+\lpf{\U}$ is $\Pi^1_2$-conservative over \ls{ACA_0^\omega}. Moreover, we proved a program extraction result for this system.
\end{itemize}

The purpose of this paper is to analyze \lpf{\Uidem} in the same way. We obtain
\begin{itemize}
\item $\ls{RCA_0^\omega} \vdash \lpf{\Uidem} \IMPL \lpf{\mu} \AND \lp{IHT}$, see \prettyref{thm:uidemiht}, and
\item $\ls{ACA_0^\omega} + \lpf{\mu} +\lp{IHT} + \lpf{\Uidem}$ is $\Pi^1_2$-conservative over $\ls{ACA_0^\omega} + \lp{IHT}$, see \prettyref{thm:main}. We also obtain a program extraction result for this system.
\end{itemize}
Many theorems from combinatorics and Ergodic theory are established using idempotent ultrafilter, see for instance \cite{HS12,vB96,vB10,BTna}. Our result provides a method to eliminate the use of idempotent ultrafilters and to extract constructive content of such proofs.

Previously idempotent ultrafilters were considered in reverse mathematics by Hirst in \cite{jH04}. He considered countable approximations of idempotent ultrafilters, similar to those we will use below. He showed that the existence of these approximations already implies \lp{IHT}. However, his concept of downward translation invariant ultrafilter is too weak to interpret iterated uses of an idempotent ultrafilter even on countably many sets. (This is due to the fact that for a downward translation invariant ultrafilter $\U_\text{appr}$ he does not investigate the structure of the set $\{\, n\in \Nat \mid X-n \in \U_\text{appr} \,\}$ used in the addition on $\beta\Nat$.)

In an excursus (\prettyref{sec:ss}) we will show that we can use the technique developed in this paper also to eliminated the stronger statement that \emph{strongly summable} ultrafilter exists. The existence of strongly summable ultrafilters is beyond ZFC.

\section{Logical Systems}

We will work in fragments of Peano arithmetic in all finite types.
The set of all finite types $\T$ is defined to be the smallest set that satisfies 
\[
0\in \T, \qquad \rho,\tau\in \T \Rightarrow \tau(\rho)\in \T
.\]
The type $0$ denotes the type of natural numbers and the type $\tau(\rho)$ denotes the type of functions from $\rho$ to $\tau$. The type $0(0)$ is abbreviated by $1$ the type $0(0(0))$ by $2$. The degree of a type is defined by
\[
\textit{deg}(0):=0 \qquad \textit{deg}(\tau(\rho)) := \max(\textit{deg}(\tau),\textit{deg}(\rho)+1)
.\]
The type of a variable or term will sometimes be written as superscript.

Equality $=_0$ for type $0$ objects will be added as a primitive notion to the systems together with the usual equality axioms.
Higher type equality $=_{\tau\rho}$ will be treated as abbreviation:
\[
x^{\tau\rho}=_{\tau\rho} y^{\tau\rho} :\equiv \Forall{z^\rho} xz =_\tau yz
.\]

Define the $\lambda$-combinators $\Pi_{\rho,\sigma}, \Sigma_{\rho,\sigma,\tau}$ for $\rho,\sigma,\tau\in\T$ to be the functionals satisfying 
\[
\Pi_{\rho,\sigma} x^\rho y^\sigma =_\rho x , \qquad \Sigma_{\rho,\sigma,\tau} x^{\tau\sigma\rho} y^{\sigma\rho}z^{\rho} =_\tau xz(yz)
.\]
Similar define the recursor  $R_\rho$ of type $\rho$ to be the functional satisfying 
\[
R_\rho 0yz =_\rho y, \qquad R_\rho (Sx^0)yz =_\rho z(R_\rho xyz)x
.\]
Let \emph{G\"odel's system $T$} be the $\T$-sorted set of closed terms that can be build up from $0^0$, the successor function $S^1$, the $\lambda$-combinators and the recursors $R_\rho$ for all finite types $\rho$.
Using the $\lambda$-combinators one easily sees that $T$ is closed under $\lambda$-abstraction, see \cite{aT73}.
Denote by $T_0$ and $T_1$ the fragments of G\"odel's system $T$, where primitive recursion is restricted to recursors $R_0$ resp.\@ $R_0$ and $R_1$.
The system $T_0$ corresponds to the extension of Kleene's primitive recursive functionals to mixed types, see \cite{sK59}, whereas full system $T$ corresponds to G\"odel's primitive recursive functionals, see \cite{kG58}.
By $T_0[F]$ we will denote the system resulting from adding a function(al) $F$ to $T_0$.
See Kohlenbach \cite{uK08} for a general introduction and more background on these systems.

The system \ls{RCA_0^\omega} is defined to be the extension of the term system $T_0$ by $\Sigma^0_1$\nobreakdash-induction,
the extensionality axioms 
\[
(\lp{E_{\rho,\tau}})\colon\Forall{z^{\tau\rho},x^\rho,y^\rho} (x=_\rho y \IMPL zx =_\tau zy)
\]
for all $\tau,\rho\in \T$, and the schema of quantifier free choice restricted to choice of numbers over functions (\lp[QF]{AC^{1,0}}), i.e.
\[
\Forall{f^1}\Exists{x^0}\lf{A_\qf}(f,x) \IMPL \Exists{F^2}\Forall{f^1} \lf{A_\qf}(f,F(f))
.\]
This schema is the higher-order equivalent to recursive comprehension (\lp[\Delta^0_1]{CA}).
(Strictly speaking the system \ls{RCA_0^\omega} was defined in \cite{uK05b} to contain only quantifier free induction instead of $\Sigma^0_1$\nobreakdash-induction. Since $\Sigma^0_1$\nobreakdash-induction is provable in that system, we may also add it directly.)
The systems \ls{WKL_0^\omega}, \ls{ACA_0^\omega} are defined to be $\ls{RCA_0^\omega}+\lp{WKL}$ resp.\@ $\ls{RCA_0^\omega}+\ls[\Pi^0_1]{CA}$.

The system \ls{RCA_0^\omega} has a functional interpretation (always combined with the elimination of extensionality and a negative translation) in $T_0$. The system \ls{ACA_0^\omega} has a functional interpretation in $T_0[\mu]$, see \cite{uK05b,AF98,uK08}.

 All of these systems are conservative over their second-order counterparts, where the second-order part is given by functions instead of sets. These second-order systems can then be interpreted in \ls{RCA_0}, resp.\@ \ls{WKL_0}, \ls{ACA_0}. See \cite{uK05b}.

We will also use the following result by Hunter.
\begin{theorem}[{\cite[Theorem~2.5]{jH08}}]\label{thm:jh}
  The system $\ls{RCA_0^\omega} + \lpf{\mu}$ is conservative over \ls{ACA_0}.
\end{theorem}
Note that this result was proven using a model construction and thus does not provide any method which would translate a proof of an analytic statement in $\ls{RCA_0^\omega} + \lpf{\mu}$ to a proof in \ls{ACA_0}. However, for $\Pi^1_2$-statements there is such a method, see \cite[Theorem~8.3.4]{AF98} and \cite{sF77}.

\begin{definition}[non-principal ultrafilter, \lpf{\U}]\label{def:ultra}
  Let \lpf{\U} be the statement that there exists a non-principal ultrafilter (on \Nat):
  \begin{equation}\label{eq:defultra}
  \lpf{\U}\colon\left\{
  \begin{aligned}
    \Exists{\U^2} \big(\ &\Forall{X^1} \left(X\in\U \OR \overline{X}\in\U\right) \\
    \AND\, &\Forall{X^1,Y^1} \left(X \cap Y \in \U \IMPL Y\in \U\right) \\
    \AND\, &\Forall{X^1,Y^1} \left(X,Y\in \U \IMPL (X\cap Y)\in\U\right)  \\
    \AND\, &\Forall{X^1} \left(X\in\U \IMPL \Forall{n}\Exists{k>n} (k\in X)\right) \\
    \AND\, &\Forall{X^1} \left(\U(X) =_0 \sg(\U(X)) =_0 \U(\lambda n. \sg(X(n)))\right)\big)
  \end{aligned}
  \right.
  \end{equation}
  Here $X\in \U$ is an abbreviation for $\U(X)=_0 0$. The type $1$ variables $X,Y$ are viewed as characteristic functions of sets, where $n\in X$ is defined to be $X(n) = 0$.
  The operation $\cap$ is defined as taking the pointwise maximum of the characteristic functions. With this, the intersection of two sets can be expressed in a quantifier free way.
  The last line of the definition states that $\U$ yields the same value for different characteristic functions of the same set and that $\U(X)\le 1$.

  For notational ease we will usually add a Skolem constant $\U$ and denote this also with \lpf{\U}.

  The second line in the definition of \lpf{\U} is equivalent to the following axiom usually found in the axiomatization of (ultra)filters:
  \[
  \Forall{X,Y} \left( X\subseteq Y \AND X\in\U \IMPL Y\in\U\right)
  .\]
  We avoided this statement in $\lpf{\U}$ since $\subseteq$ cannot be expressed in a quantifier free way.
\end{definition}

\begin{definition}[idempotent ultrafilter, \lpf{\Uidem}]
  An idempotent ultrafilter is a non-principal ultrafilter  $\U$ such that
  \begin{equation}\label{eq:idemex}
    \Forall{X^1} \left(X\in\U \IMPL \left\{\, n\in\Nat \mid X - n \in \U \,\right\} \in \U \right)
    .
  \end{equation}
  Let \lpf{\Uidem} be that statement that an idempotent ultrafilter exists, i.e.~\eqref{eq:defultra} where $\U$ is also required to satisfy \eqref{eq:idemex}.
\end{definition}

\section{Iterated Hindman's theorem}

Let $X$ be a finite or infinite set of natural numbers and $(x_i)_i$ be a strictly ascending enumeration of it.
We will write 
\[
\FS{X} = \FS{(x_i)_i} := \{\, x_{i_k} + \dots + x_{i_1} \mid i_1 < i_2 < \dots < i_k \,\}
\]
for the set of finite sums of $X$.

\begin{definition}[\cite{nH74}]
  \emph{Hindman's theorem} (\lp{HT}) is the statement that for each coloring $c\colon \Nat \longrightarrow 2$ of the natural numbers there exists an infinite set $X$ such that $\FS{X}$ is homogeneous for $c$. 
\end{definition}

Hindman's theorem is implied by \ls{ACA_0^+} that is \lp{ACA_0} plus the statement that for each set $X$ the $\omega$-Turing jump $X^{(\omega)}$ exists, and it implies \ls{ACA_0}, see \cite{BHS87}. 
It is open whether \lp{HT} is equivalent to \ls{ACA_0} or \ls{ACA_0^+} or whether it lies strictly between, see \cite{aM11}.

\begin{definition}[\cite{BHS87,jH04}]
  \emph{Iterated Hindman's theorem} (\lp{IHT}) is the statement that for each sequence of colorings $c_k\colon \Nat \longrightarrow 2$ there exists a strictly ascending sequence $(x_i)_{i\in\Nat}$ such that for each $k$ the set $\FS{(x_i)_{i=k}^\infty}$ is homogeneous for $c_k$.
\end{definition}

In \cite[Theorem~4.13]{BHS87} it was also shown that \lp{IHT} is provable in \ls{ACA_0^+}.

\begin{theorem}[\ls{RCA_0}]\label{thm:htps}\mbox{}
  \begin{enumerate}[label=(\roman*)]
  \item\label{enum:1:1} Hindman's theorem is equivalent to the statement that for each infinite set $Y$ and each coloring $c\colon \FS{Y} \longrightarrow 2$ there exists an infinite subset $X\subseteq \FS{Y}$ such that $\FS{X}$ is homogeneous for $c$.
  \item\label{enum:1:2} Similarly, iterated Hindman's theorem is equivalent to the statement that for each infinite set $Y$ and for each sequence of colorings $c_k\colon \FS{Y} \longrightarrow 2$ there exists a sequence $(x_i)_{i\in\Nat}\subseteq \FS{Y}$ which satisfies the conclusion of \lp{IHT}.
  \end{enumerate}
\end{theorem}
\begin{proof}
  \ref{enum:1:1} follows from Lemma~2.1 of \cite{BH93}, see also \cite{vB10}, and noting that the proof of the equivalences formalizes in \ls{RCA_0}.
  \ref{enum:1:2} follows by iterating the construction of \ref{enum:1:1}.
\end{proof}

\begin{theorem}\label{thm:uidemiht}
  $\ls{RCA_0^\omega} \vdash \lpf{\U_\textrm{idem}} \IMPL \lp{IHT}$
\end{theorem}
To prove this theorem we will use the following notation and lemma.
For an $X\subseteq \Nat$ we set $X^\star:= \{\, n \in X \mid X-n\in \U \,\}$. It is easy to see that if $\U$ is idempotent we have that $X\in \U$ implies that $X^\star\in \U$.
\begin{lemma}[$\ls{RCA_0^\omega}+\lpf{\Uidem}$, {\cite[Lemma~4.14]{HS12}}]\label{lem:xstar}
  Let $\U$ be an idempotent ultrafilter.
  For each $X\in \U$ and each $n\in X^\star\in \U$ also $X^\star-n\in \U$.
\end{lemma}
\begin{proof}
  Let $Y:= X-n$. By assumption $Y\in \U$ and therefor also $Y^\star\in \U$. 
  We claim that $Y^\star \subseteq X^\star -n$.
  To see this, let $m$ be an arbitrary element of $Y^\star$. Then $m\in Y$ and so $m+n\in X$. Also $Y-m\in \U$, therefore $X-(m+n)\in \U$. From this follows that $m+n\in X^\star$ and with this the claim.
  Since $Y^\star$ is in $\U$, the claim implies that $X^\star - n\in \U$.
\end{proof}

\begin{proof}[Proof of \prettyref{thm:uidemiht}]
  Let $\U$ be an idempotent ultrafilter and let $c_i$ be a sequence of colorings. 
  Set $A_i$ to be a set such that $c_i$ is constant on $A_i$ and $A_i\in \U$. For instance one may take either $c_i^{-1}(0)$ or $c_i^{-1}(1)$.

  We will recursively build a sequence $(x_j)_{j\in \Nat}$ such that $\FS{(x_j)_{j=i}^\infty} \subseteq  A_i^\star$ for all $i$.
  Assume that we have chosen $(x_j)_{j=0}^k$ such that $\FS{(x_j)_{j=i}^k} \in A_i^\star$ for each $i$.
  Let $B_i:= \bigcap_{n\in\FS{(x_j)_{j=i}^k} \cup \{0\}} A_i^\star - n$. This is a finite intersection of---by \prettyref{lem:xstar}---sets in $\U$. Thus, $B_i\in \U$ and in particular $\bigcap_{i \le k+1} B_i$ is in $\U$ and therefore not empty. Let $x_{k+1}$ be an element of this set.
  Then $x_{k+1} + \FS{(x_j)_{j=i}^k}  \subseteq  A_i^\star$. Hence, $\FS{(x_j)_{j=i}^{k+1}}  \subseteq  A_i^\star$ for each $i$.
\end{proof}

The main results of this paper are the following theorems.
\begin{theorem}\label{thm:main}
  The system $\ls{ACA_0^\omega} + \lpf{\mu} + \lp{IHT} + \lpf{\Uidem}$ is $\Pi^1_2$\nobreakdash-conservative over $\ls{ACA_0^\omega} + \lp{IHT}$ and thus in particular over $\ls{ACA_0}+\lp{IHT}$ and \lp{ACA_0^+}.
\end{theorem}

\begin{theorem}[program extraction]\label{thm:pe}
  Let $\Forall{f}\Exists{g}\lf{A}(f,g)$ be a $\Pi^1_2$-sentence. If
  \begin{equation}\label{eq:pe}
  \ls{ACA_0^\omega} + \lpf{\mu} + \lp{IHT} + \lpf{\U_\textrm{idem}} \vdash \Forall{f} \Exists{g} \lf{A}(f,g)
  \end{equation}
  then one can extract from a proof a term $t\in T_1[\mu]$ such that
  \[
  \Forall{f} \Exists{g} \lf{A}(f,tf)
  .\]
\end{theorem}

\prettyref{thm:main} is optimal in the sense that \lpf{\U_\textrm{idem}} cannot be $\Pi^1_2$-conservative over any system not containing \lp{IHT} because of \prettyref{thm:uidemiht}. The program extraction of \prettyref{thm:pe} is not faithful. For type reason \lpf{\U_\textrm{idem}} does not imply the totality of $R_1$. However, we do not have a faithful functional interpretation for \lp{IHT} yet and just use $R_1$ and $\mu$ to emulate \ls{ACA_0^+}.

The strategy of the proofs of these theorems is similar to the strategy in  \cite{aK12b}.
We will proceed roughly in the following steps.

\begin{enumerate}
\item 
  Let $\Forall{f} \Exists{g} \lf{A}(f,g)$ be a $\Pi^1_2$-statement
  such that \[
  \ls{ACA_0^\omega} + \lpf{\mu} + \lp{IHT} + \lpf{\U_\textrm{idem}} \vdash \Forall{f} \Exists{g} \lf{A}(f,g)
  .\]
  Using the functional interpretation and a program normalization we show that each application of $\U$ in this proof is of the form $\U(t[n^0])$ for a term $t$ that contains only $n$ free and with $\lambda n .t \in T_0[\U]$. (This step does not differ from the first step in \cite{aK12b}.)
\item We now construct provably in $\ls{RCA_0^\omega} + \lpf{\mu} + \lp{IHT}$ a---so called---downward translation partial ultrafilter, which acts like an idempotent ultrafilter on the sets that occur in the proof. The idempotent ultrafilter is then replaced by this object in the proof.
\item Applying \prettyref{thm:jh} to this yields \prettyref{thm:main}.
\item For \prettyref{thm:pe} we notice that in $T_1[\mu]$ we can define a functional solving \lp{IHT} and thus we can explicitly describe the downward translation partial ultrafilter we construction in the second step. With this we get a program witnessing $g$. 
\end{enumerate}

\section{Downward translation partial ultrafilter}

In \cite{aK12b} we built a---so called---partial non-principal ultrafilter which acted on the algebra of sets that were used in a proof like a non-principal ultrafilter. We will now briefly recall the notions of algebra and partial non-principal ultrafilter. After this, we will introduce the notions of \emph{downward translation algebra} and \emph{downward translation partial ultrafilter} which will be suitable for handling idempotent ultrafilters.
\begin{definition}\mbox{}
  \begin{itemize}
  \item   An \emph{algebra} is a set $\mathcal{A} \subseteq \ps{\Nat}$ that is closed under complement, finite unions, and finite intersection.
  \item For an algebra $\mathcal{A}$ we call a set $\F$ a \emph{partial non-principal ultrafilter} for $\mathcal{A}$ if $\F$ satisfies the axioms for a non-principal ultrafilter relativized to $\mathcal{A}$, i.e.
    \begin{equation}\label{eq:pnu}
      \left\{
        \begin{aligned}
          &\Forall{X\in \mathcal{A}} \left(X\in\F \OR \overline{X}\in\F\right) \\
          \AND\, &\Forall{X,Y\in \mathcal{A}} \left(X \cap Y \in \F \IMPL Y\in \F\right) \\
          \AND\, &\Forall{X,Y\in \mathcal{A}} \left(X,Y\in \F \IMPL (X\cap Y)\in\F\right)  \\
          \AND\, &\Forall{X\in\mathcal{A}} \left(X\in\F \IMPL \Forall{n}\Exists{k>n} k\in X\right) \\
          \AND\, &\Forall{X^1} \left(\F(X) =_0 \sg(\F(X)) =_0 \F(\lambda n . \sg(X(n))\right).
        \end{aligned}
      \right.
    \end{equation}
    (Note that we do not require $\F$ to be a subset of $\mathcal{A}$ as we did in \cite{aK12b}. This restriction was actually not used in \cite{aK12b} and could have been omitted.)
  \end{itemize}
\end{definition}

\begin{definition}\mbox{}
  \begin{itemize}
  \item An algebra $\mathcal{A}$ is called \emph{downward translation algebra} if it is closed under downward translations, i.e.
    \[
    X\in\mathcal{A} \Rightarrow  \Forall{n\in\Nat} \left( X-n \in \mathcal{A}\right)
    .\]
  \item A \emph{downward translation partial ultrafilter} is a partial non-principal ultrafilter $\F$ for a downward translation algebra $\mathcal{A}$ which in addition to \eqref{eq:pnu} satisfies the following axiom
    \begin{equation}\label{eq:dta}
      \Forall{X\in \mathcal{F}} \left(\left\{ n \in \Nat \sizeMid  X-n\in \F \right\} \in \F\right)
      .
    \end{equation}
    In other words, a downward translation partial ultrafilter for $\mathcal{A}$ is an object which satisfies the axioms \lpf{\Uidem} but where $X,Y$ is restricted to $\mathcal{A}$.
  \end{itemize}
\end{definition}

Like in \cite{aK12b} we will mostly work with countable downward translation algebras $\mathcal{A}$ which are given by a sequence of sets $(A_i)_{i\in\Nat}$. The characteristic function of $\chi_\mathcal{A}$ of $\mathcal{A}$ is then given by 
\[
\chi_{\mathcal{A}}(X)=
\begin{cases}
  0 & \text{if\, $\Exists{i} \left(A_i=X\right)$,} \\
  1 & \text{otherwise.}
\end{cases}
\]
Such a characteristic function can be defined using $\mu$.

It is easy to see that in \ls{RCA_0^\omega} each sequence of sets $(A_i)_{i\in\Nat}$ can be extended to form a countable downward translation algebra.

The downward translation partial ultrafilters we will build will be of the following form
\[
\FFS{(x_i)_i} := \left\{\, X \subseteq \Nat \sizeMid \Exists{m} \FS{(x_i)_{i=m}^\infty} \subseteq X \,\right\} 
\]
where $(x_i)_{i\in\Nat}$ is a strictly ascending sequence of natural numbers.

One checks that $\FFS{(x_i)_i}$ is closed under finite intersections, taking supersets, and contains only infinite sets. Thus, it is a filter.

\begin{lemma}\label{lem:dtuup}
  Let $(x_i)_{i\in\Nat}$ be an ascending sequence of natural numbers.
  Then $\FFS{(x_i)_i}$ satisfies \eqref{eq:dta}.

  In particular, if $\FFS{(x_i)_i}$ is a partial non-principal ultrafilter for a downward translation algebra $\mathcal{A}$ then it is already a downward translation partial ultrafilter for $\mathcal{A}$.
\end{lemma}
\begin{proof}
  Let $X\in \FFS{(x_i)_i}$. By definition there is an $m$, such that $\FS{(x_i)_{i=m}^\infty} \subseteq X$.

  It is sufficient to show that $\FS{(x_i)_{i=m}^\infty} \subseteq \left\{ n \in \Nat \sizeMid  X-n\in \FFS{(x_i)_i} \right\}$ or in other words that
  for each $n\in \FS{(x_i)_{i=m}^\infty}$ we have $X-n\in \FFS{(x_i)_i}$.

  Indeed each $n\in \FS{(x_i)_{i=m}^\infty}$ can be written as $n= x_{i_k}+x_{i_{k-1}}+\dots+x_{i_1}$ for $i_k>i_{k-1}>\dots>i_1\ge m$. Let now $l:=i_k + 1$. Then for each $n'\in \FS{(x_i)_{i=l}^\infty}$ the number $n'+n$ is an element of $\FS{(x_i)_{i=m}^\infty}$ or in other words $\FS{(x_i)_{i=l}^\infty} \subseteq \FS{(x_i)_{i=m}^\infty} - n$. Thus, $\FS{(x_i)_{i=l}^\infty} \subseteq X-n$, with this $X-n\in \FFS{(x_i)_i}$,  and the lemma follows.
\end{proof}

\begin{lemma}\label{lem:ss}
  Let $\mathcal{A}$ be a downward translation algebra. If for a sequence $(x_i)_{i\in\Nat}$ the set $\FFS{(x_i)_i}$  is a downward translation partial ultrafilter for $\mathcal{A}$ then for any sequence $(y_i)_{i\in\Nat}$ with $\FS{(y_i)_i} \subseteq \FS{(x_i)_i}$ we have
  $\FFS{(x_i)_i} \cap \mathcal{A} = \FFS{(y_i)_i} \cap \mathcal{A}$.
\end{lemma}
\begin{proof}
  By definition of $\FFS{(x_i)_i}$ we have that $\FFS{(x_i)_i} \subseteq \FFS{(y_i)_i}$.
  Moreover, the set $\FFS{(y_i)_i}$ is a filter and, therefore, contains for each $X$ at most one of $X$ and $\overline{X}$.
  Now $\FFS{(x_i)_i}$ is maximal in $\mathcal{A}$ in the sense that for each $X\in\mathcal{A}$ either $X$ or $\overline{X}$ is an element in $\FFS{(x_i)_i}$. Thus, both filters must be equal on $\mathcal{A}$.
\end{proof}

\begin{theorem}\label{thm:dtex}
  Let $\mathcal{A}$ be a countable downward translation algebra and let $(x_i)_{i\in\Nat}$ be a sequence such that $\FFS{(x_i)_i}$ is a downward translation partial ultrafilter for $\mathcal{A}$. Then $\ls{RCA_0^\omega} + \lpf{\mu} + \lp{IHT}$ proves that for each countable downward translation algebra $\mathcal{\tilde{A}} = (\tilde{A}_i)_{i\in\Nat} \supseteq \mathcal{A}$ there exists a sequence $(y_i)_{i\in\Nat}$ with $\FS{(y_i)_i)} \subseteq \FS{(x_i)_i}$, such that $\FFS{(y_i)_i}$ is a downward translation partial ultrafilter for $\mathcal{A}$.

In particular,  $\FFS{(y_i)_i} \supseteq \FFS{(x_i)_i}$ and $\FFS{(y_i)_i} \cap \mathcal{A} = \FFS{(x_i)_i} \cap \mathcal{A}$.
\end{theorem}
\begin{proof}
  Let
  \[
  c_k(x) :=
  \begin{cases}
    0 & \text{if $x\in \tilde{A}_k$,} \\
    1 & \text{if $x\notin \tilde{A}_k$.}
  \end{cases}
  \]
  By \lp{IHT} and \prettyref{thm:htps} there exists a sequence $(y_i)_i$ with $\FS{(y_i)_i} \subseteq \FS{(x_i)_i}$ such that 
  $\FS{(y_i)_{i=k}^\infty}$ is  homogeneous for $c_k$, i.e.\@ $\FS{(y_i)_{i=k}^\infty}$ is contained in either $\tilde{A}_k$ or $\overline{\tilde{A}_k}$. Thus, for each $k$ the filter $\FFS{(y_i)_i}$ contains either $\tilde{A}_k$ or $\overline{\tilde{A}_k}$ and is, therefore, a downward translation partial ultrafilter for $\mathcal{\tilde{A}}$.
\end{proof}

\begin{remark}[Stone-\v{C}ech compactification]
  The sets of filters of the form $\FFS{(x_i)_i}$ can be viewed as the following closed sets of the Stone-\v{C}ech compactification $\beta\Nat$:
  \[
  \bigcap_{m=1}^\infty \{ \U \in \beta\Nat \mid \FS{(x_i)_{i=m}^\infty} \in \U \}
  \]
  The use of such closed sets is inspired by \cite[Theorem~2.5]{vB10}, \cite[Lemma~5.11]{HS12} where it is shown that each such set contains an idempotent ultrafilter.
\end{remark}

\section{Proof theory}\label{sec:prooftheory}

We proceed like in \cite{aK12b}.
The elimination of extensionality \cite[Lemma~7]{aK12b} is also applicable to \lpf{\U_\textrm{idem}} instead of \lpf{\U} since 
\[ 
\lpf{\U_\textrm{idem}} \equiv \lpf{\U} \text{ extended by } \Forall{X^1} \left(X\in\U \IMPL \left\{\, k\in\Nat \mid X - k \in \U \,\right\} \in \U \right)
\]
and the added axiom contains only quantification over variables of degree $\le 1$.
Therefore, it is not changed by the elimination of extensionality translation.
We obtain the following lemma.
\begin{lemma}[elimination of extensionality, cf.~{\cite[Section~10.4]{uK08}}]\label{lem:eliex}
  If $\lf{A}$ is a sentence that contains only quantification over variables of degree $\le 1$ and 
  \[
  \ls{RCA_0^\omega} \vdash \lpf{\Uidem} \IMPL \lf{A}
  \]
  then
  \[
  \WEPAw + \lp[QF]{AC^{1,0}} \vdash \lpf{\Uidem} \IMPL \lf{A}
  .\]
  Here $\WEPAw + \lp[QF]{AC^{1,0}}$ is the weakly extensional counterpart to \ls{RCA_0^\omega}.

  Since $\mu$ is provably extensional this lemma remains true if one adds $\lpf{\mu}$ to both systems.
\end{lemma}

We will also use the following term normalization result.

\begin{theorem}[term normalization for degree $2$, {\cite[Theorem~8]{aK12b}, \cite{uK99,AF98}}]\label{thm:cut2}
  Let $F_1,\dots,F_n$ be constants of degree~$\le 2$.

  For every term $t^1\in T_0[F_0,\dots,F_{n-1}]$ there is a term $\tilde{t}\in T_0[F_0,\dots,F_{n-1}]$ with
  \[
  \WEPAw \vdash t =_1 \tilde{t}
  \]
  and such that
  every occurrence of an $F_i$ in $\tilde{t}$ is of the form
  \[
  F_i(\tilde{t}_0[y^0], \dots, \tilde{t}_{k-1}[y^0])
  .\]
  Here $k$ is the arity of $F_i$, and $\tilde{t}_j[y^0]$ are fixed terms whose only free variable is~$y^0$.
\end{theorem}

The axiom \lpf{\Uidem} can be prenexed to the following statement.
\begin{align*}
  \Exists{\U^2}\Forall{X^1,Y^1}\Forall{n}\Exists{k}  \big(\ &\Forall{X} \left(X\in\U \OR \overline{X}\in\U\right) \\
  \AND\, & \left(X \cap Y \in \U \IMPL Y\in \U\right) \\
  \AND\, & \left(X,Y\in \U \IMPL (X\cap Y)\in\U\right)  \\
  \AND\, & \left(X\in\U \IMPL (k>n \AND  k\in X)\right) \\
  \AND\, & \left(X\in\U \IMPL \left\{\, n\in\Nat \mid X - n \in \U \,\right\} \in \U \right) \\
  \AND\, & \left(\U(X) =_0 \sg(\U(X)) =_0 \U(\lambda n. \sg(X(n)))\right)\big)
\end{align*}
By coding the sets $X,Y$ together we obtain the following
\[
\Exists{\U^2}\Forall{Z^1}\Forall{n}\Exists{k} \lpf{\Uidem}_\qf(\U,Z,n,k)
\]
where $\lpf{\Uidem}_\qf$ is the quantifier free matrix of the above statement.

An application of \lp[QF]{AC^{1,0}} yields
\begin{equation}\label{eq:uidem}
  \Exists{\U^2}\Exists{K^2}\Forall{Z}\Forall{n} \lpf{\Uidem}_\qf(\U,Z,n,KnZ)
\end{equation}

The variable $K$ may always be chosen to be the following functional definable using $\mu$.
\begin{equation}\label{eq:k'}
  K'(n,X) :=
  \begin{cases}
    \min\{ k \in X \mid k>n\} & \text{if exists,} \\
    0 & \text{otherwise.}
  \end{cases}
\end{equation}
Therefore, the real difficulty is the construction of a suitable $\U$.

\begin{proof}[Proof of \prettyref{thm:main}]
  Let $\Forall{f}\Exists{g}\lf{A}(f,g)$ a $\Pi^1_2$-statement not containing $\U$ or $\mu$ and provable in $\ls{ACA_0^\omega} + \lpf{\mu} + \lp{IHT} + \lpf{\Uidem}$.
  Since \lpf{\Uidem} implies \lpf{\mu}, \lp{ACA_0}, and \lp{IHT} we arrive at
  \[
  \ls{RCA_0^\omega} + \lpf{\Uidem} \vdash \Forall{f} \Exists{g} \lf{A}(f,g)
  .\]
  Using $\mu$ one can find a quantifier free formula $\lf{A'_\qf}(f,g)$ such that $\lf{A'_\qf}$ does not contain $U$ and $\lf{A}(f,g) \IFF \lf{A'_\qf}(f,g)$. Together with the deduction theorem we arrive at the following.
  \[
  \ls{RCA_0^\omega} + \lpf{\mu} \vdash \lpf{\Uidem} \IMPL  \Forall{f} \Exists{g} \lf{A'_\qf}(f,g)
  \]
  Applying the elimination of extensionality we get
  \[
  \WEPAw + \lp[QF]{AC^{1,0}} + \lpf{\mu} \vdash \lpf{\Uidem} \IMPL  \Forall{f} \Exists{g} \lf{A'_\qf}(f,g)
  .\]
  After reintroducing a variable $\U$ for the ultrafilter and using \eqref{eq:uidem} we obtain
  \[
  \left(\Exists{\U^2}\Exists{K^2}\Forall{Z}\Forall{n} \lpf{\Uidem}_\qf(\U,Z,n,KnZ)\right)\IMPL  \Forall{f} \Exists{g} \lf{A'_\qf}(f,g)
  .\]
  Pulling out the quantifiers we get
  \[
  \Forall{f}\Forall{\U^2}\Forall{K^2} \Exists{Z^1} \Exists{n} \Exists{g}
  \lpf{\Uidem}_\qf(\U,Z,n,KnZ)\IMPL\lf{A'_\qf}(f,g)
  .\]
  A functional interpretation yields $t_Z,t_n,t_g \in T_0[\mu,\U,K,f]$ such that
  \begin{equation}\label{eq:1}
    \WEPAw + \lpf{\mu} \vdash \Forall{f}\Forall{\U^2}\Forall{K^2} \lpf{\Uidem}_\qf(\U,t_Z,t_n,Kt_nt_Z)\IMPL\lf{A'_\qf}(f,t_g)
  .\end{equation}

  Applying \prettyref{thm:cut2} to $t_Z,t_n,t_g$ we obtain terms  $t_Z',t_n',t_g'$, which are provably equal and where every occurrence of $\U, K$ is of the form
  \[
  \U(t[j^0])\quad\text{resp.}\quad K(n^0,t[j^0])
  \]
  for a $t\in T_0[\mu,\U,K,f]$.
  
  Let $(t_i)_{i<n}$ be the list of all of these terms $t$ to which $\U$ and $K$ are applied. Assume that this list is partially ordered according to the subterm ordering, i.e.\@ if $t_i$ is a subterm of $t_j$ then $i<j$.

  We will now build for each $f$ a downward translation partial ultrafilter $\F$, which acts on these occurrences like a real idempotent non-principal ultrafilter. For this fix an arbitrary $f$.

  The filter $\F$ is build by iterated applications of \prettyref{thm:dtex}: We start with the trivial downward translation algebra 
  \[
  \mathcal{A}_{-1} := \{\, X \subseteq \Nat \mid \text{$X$ is finite or cofinite} \,\}
  \]
  and the Fr\`{e}chet filter 
  \[
  \F_{-1} :=  \{\, X \subseteq \Nat \mid \text{$X$ is cofinite} \,\} = \FFS{(i)_{i\in\Nat}}
  .\]
  It is clear that $\F_{-1}$ is a partial non-principal ultrafilter for $\mathcal{A}_{-1}$. By \prettyref{lem:dtuup} it is also a downward translation partial ultrafilter for $\mathcal{A}_{-1}$.

  Assume now that $\mathcal{A}_{i-1}$ and $\F_{i-1}$ are already defined.
  Let $\mathcal{A}_i$ be the downward translation algebra spanned by $\mathcal{A}_{i-1}$ and the sets described by $t_i$ where $\U,K$ are replaced by $\F_{i-1}$ and $K'$ from \eqref{eq:k'}, i.e.\@ ${\big(t_i[\U/ \F_{i-1},K/K'] (j)\big)}_{j\in\Nat}$. Let $\F_i=\FFS{(y_k)_k}$ be the extension of $\F_{i-1}=\FFS{(x_k)_k}$ to the new downward translation algebra $\mathcal{A}_i$ as constructed in \prettyref{thm:dtex}.

  Since $\U$ is in $t_i$ only applied to subterms of $t_i$ we obtain by the construction of the filter and \prettyref{lem:ss} that
  \[
  t_i[\U/ \F_{j-1},K/K'] =_1 t_i[\U/ \F_{j'-1},K/K'] \qquad\text{for all $j,j'>i$}
  .\]

  For the resulting downward translation partial ultrafilter $\F:=\F_n$ we obtain that
  \[
  \lpp{(\Uidem)_\qf}{\F,t_Z[\F,K',f] ,t_n[\F,K',f],K't_n[\F,K',f]t_Z[\F,K',f]}
  .\]
  In total we get
  \begin{multline*}
  \ls{RCA_0^\omega} + \lpf{\mu} + \lp{IHT} \vdash \\ \Forall{f} \Exists{\F} \lpp{(\Uidem)_\qf}{\F,t_Z[\F,K',f] ,t_n[\F,K',f],K't_n[\F,K',f]t_Z[\F,K',f]}
  .\end{multline*}

  Combining this with \eqref{eq:1} we get 
  \[
  \ls{RCA_0^\omega} + \lpf{\mu} + \lp{IHT} \vdash \Forall{f} \Exists{\F} \lf{A'_\qf}(f,t_g[\mu,\F,K,f])
  \]
  and hence
  \begin{equation}\label{eq:2}
    \ls{RCA_0^\omega} + \lpf{\mu} + \lp{IHT} \vdash \Forall{f} \Exists{g} \lf{A'_\qf}(f,g)
  .\end{equation}
  Replacing $\lf{A'_\qf}$ with the equivalent formula $\lf{A}$ which does not contain $\mu$ we obtain
  \[
  \ls{RCA_0^\omega} + \lpf{\mu} + \lp{IHT} \vdash \Forall{f} \Exists{g} \lf{A}(f,g)
  .\]
  Noting that \lp{IHT} is analytic and applying \prettyref{thm:jh} we get
  \[
  \ls{ACA_0}+ \lp{IHT} \vdash \Forall{f} \Exists{g} \lf{A}(f,g)
  .\qedhere \]
\end{proof}

\begin{proof}[Proof of \prettyref{thm:pe}]
  In $T_1[\mu]$ one can define a functional which maps a set $X$ to its $\omega$-Turing jump. Therefore, the functional interpretation of \ls{ACA_0^+} can be solved in $T_1[\mu]$. Since $\ls{ACA_0^+} \vdash \lp{IHT}$, we can find a solution of the functional interpretation of \eqref{eq:2}. This implies \prettyref{thm:pe}.
\end{proof}

\begin{remark}
  Theorems \ref{thm:main} and \ref{thm:pe} remain true if \lp{\Uidem} is replaced by the statement that an idempotent ultrafilter for a countable semigroup $G$ exists. This follows simply from the fact that Hindman's theorem and iterated Hindman's theorem for $\Nat$ imply their variants for any countable semigroup, see \cite[Lemma~2.1]{BH93}, and the fact that we did not use any property of the natural numbers. In fact, we carefully formulated the definitions and proofs such that we did not even use commutativity of addition.
\end{remark}
\begin{remark}
  \prettyref{thm:pe} remains true if one replaces \eqref{eq:pe} by the following.
  \begin{multline*}
  \ls{ACA_0^\omega} + \lpf{\mu} + \lp{IHT} \vdash \\
  \Forall{f} \big(\Exists \U \,\text{[$\U$ is an idempotent ultrafilter extending $t_\F(f)$]} \IMPL \Exists{g} \lf{A}(f,g)\big)
  \end{multline*}
  where $t_\F$ is a closed term such that $t_\F(f)$ codes a downward translation partial ultrafilter.
  A similar statement also holds for \prettyref{thm:main}.

  This follows by taking $t_\F(f)$ instead of the trivial filter for $\F_{-1}$ in the proof of \prettyref{thm:main}.
\end{remark}

\section{Strongly Summable Ultrafilters}\label{sec:ss}

A strongly summable ultrafilter is an ultrafilter $\U$ such that for each $X\in \U$ there exists a strictly ascending sequence $(x_i)_{i\in\Nat}$ such that $\FS{(x_i)_i}\subseteq X$ and $\FS{(x_i)_i} \in \U$.

It is known that each strongly summable ultrafilter is idempotent. However, the reverse is not true---while each set in an idempotent ultrafilter contains a set of the form $\FS{(x_i)_{i\in \Nat}}$ in general this set is not contained in the ultrafilter.
The existence of strongly summable ultrafilters is not provable in ZFC\@. The existence follows for instance from Martin's Axiom.
See Chapter~12 of \cite{HS12} for details.

We will now show how to modify the above proofs to obtain conservativity and program extraction for strongly summable ultrafilters.
Let $\lpf{\Uss}$ be that statement that a strongly summable ultrafilter exists, i.e.~\eqref{eq:defultra} from \prettyref{def:ultra} plus the requirement
\[
\Forall{X} \left(X\in \U \IMPL \Exists{(x_i)_{i\in\Nat}} \left(\FS{(x_i)_i}\in \U\right)\AND \FS{(x_i)_i} \subseteq X\right)
.\]

\begin{theorem}
  The Theorems~\ref{thm:main} and \ref{thm:pe} hold with \lpf{\Uidem} replaced by \lpf{\Uss}.
\end{theorem}
\begin{proof}
  We proceed like in the original proof for idempotent ultrafilter.
  After the application of \prettyref{lem:eliex} we get
  \[
  \WEPAw + \lp[QF]{AC^{1,0}} \vdash \lpf{\Uss} \IMPL \lf{A}
  .\]
  
  Here we slightly diverge from the original proof and strength \lpf{\Uss} by adding a uniform functional which yields the finite sum set, i.e.
  \begin{equation}\label{eq:uss}\left\{
      \begin{aligned}
        \Exists{\U^2}\Exists{\mathcal{I}^2} \big(\ &\Forall{X} \left(X\in\U \OR \overline{X}\in\U\right) \\
        \AND\, &\Forall{X^1,Y^1} \left(X \cap Y \in \U \IMPL Y\in \U\right) \\
        \AND\, &\Forall{X^1,Y^1} \left(X,Y\in \U \IMPL (X\cap Y)\in\U\right)  \\
        \AND\, &\Forall{X^1} \left(X\in\U \IMPL \Forall{n}\Exists{k>n} (k\in X)\right) \\
        \AND\, &\Forall{X^1} \left(X\in\U \IMPL \left(\FS{\mathcal{\hat{I}}(X)}\in \U \AND \Forall{n\in\FS{\mathcal{\hat{I}}(X)}}\; n\in X \right)\right) \\
        \AND\, &\Forall{X^1} \left(\U(X) =_0 \sg(\U(X)) =_0 \U(\lambda n. \sg(X(n)))\right)\big)
      \end{aligned}
    \right.
  \end{equation}
  where $\mathcal{\hat{I}}(x) := \max(\mathcal{I}(x), \max_{x'<x}(\mathcal{\hat{I}}(x')+1))$ ensures that it describes a strictly increasing enumeration.

  After prenexation and the application of \lp[QF]{AC} we see that this statement is equivalent to a statement of the following form, cf.~\eqref{eq:uidem}.
  \[
  \Exists{\U^2}\Exists{\mathcal{I}^2}\Exists{K^2}\Forall{Z}\Forall{n} \lpf{\Uss}_\qf(\U,\mathcal{I},Z,n,KnZ)
  .\]
  
  Now the only thing we have to take care of is to build an approximation to the functional $\mathcal{I}$ along with the construction of the approximation for $\U$.
  Let $(t_i)_{i<j}$ be the list of terms to which $\U$, $\mathcal{I}$ and $K$ are applied to after normalizing the extracted terms. Again we assume that these terms are ordered with respect to the subterm ordering.

  Let $\mathcal{E}_i(X)$ be the functional which extracts the maximal part contained in $X$ of the generating sequence $(x_k)_{k\in\Nat}$ of $\F_i=\FFS{(x_k)_k}$ if $X\in \F_i$, i.e.\@  
  \[
  \mathcal{E}_i(X) :=
  \begin{cases}
    (x_k)_{k=m}^\infty &
    \begin{varwidth}{10cm}
      if $X\in \F_i=\FFS{(x_k)_k}$ and \\ $m$ is minimal satisfying $\FS{(x_k)_{k=m}^\infty}\subseteq X$,
    \end{varwidth} \\
    (0)_k & \text{otherwise.}
  \end{cases}
  \]
  Then we will set 
  \[
  \mathcal{I}_i(X) := 
  \begin{cases}
    \mathcal{E}_j(X) & \text{where $j\le i$ minimal with $X\in \F_j$ if such a $j$ exists,} \\
    (0)_k & \text{otherwise.}
  \end{cases}
  \]
  It is clear that $\mathcal{E}_i$ and $\mathcal{I}_i$ can be defined with the help of $\mu$. It is also clear that $\mathcal{I}_i(X) = \mathcal{I}_{i'}(X)$ if $\F_i(X) = \F_{i'}(X)$.

  Now we can proceed to build the filters $\F_i$ as we did in the original proof. The functional $\mathcal{I}$ can be interpreted in each step using $\mathcal{I}_i$. The only thing we change in the generation of $\F_i$ is that the algebras $\mathcal{A}_i$ will be extended by the sets $\mathcal{I}_i(\F_i)$ after each step to ensure that the finite sum subsets are included. By construction $\F_i$ remains a downward translation partial ultrafilter for the extension of $\mathcal{A}_i$. Thus, $\F_i$ and $\mathcal{I}_i$ prove \eqref{eq:uss} relativized to this algebra.
\end{proof}

\subsection*{Concluding remarks}

In \cite{hTna} Towsner also considers ultrafilters in reverse mathematics. He works with second order systems and formalizes an ultrafilter as a predicate on the second order sort of the system. This extension of \ls{ACA_0} is denoted by $\ls{ACA_0} + \exists\mathfrak{U}$. It is easy to see that $\ls{ACA_0} + \exists\mathfrak{U}$ can be embedded into $\ls{ACA_0^\omega} + \lpf{\U}$.
It is not known whether $\ls{ACA_0^\omega} + \lpf{\U}$ is conservative over $\ls{ACA_0} + \exists\mathfrak{U}$. The techniques usually used to show that finite type systems are conservative over second order systems (like the interpretation in \ls{ECF}) require continuity and $\U$ obviously is not continuous.
However, it might be possible to adapt Hunter's proof (\prettyref{thm:jh}) to obtain conservativity.

Towsner asked whether $\ls{ACA_0} + \exists\mathfrak{U} + \text{``every element of $\mathfrak{U}$ is an IP-set''}$ does imply \lp{ACA_0^+} (Question~4.4). Since this system implies \lp{IHT} and idempotent ultrafilters contain only IP-sets, \prettyref{thm:main} reduces this question to the open problem whether \lp{IHT} implies \lp{ACA_0^+}. Also, Towsner asks over which theories of reverse mathematics the existence of an idempotent ultrafilter is conservative (Question~4.6). \prettyref{thm:main} answers this question to a large extend.

\bibliographystyle{amsplain}
\bibliography{../bib}

\end{document}